\documentclass{amsart}

\usepackage{amsmath,amssymb,amsthm}
\usepackage{fullpage}

\newtheorem{theorem}{Theorem}
\newtheorem{lemma}{Lemma}
\newtheorem{proposition}{Proposition}

\title{(Claw, $\vec C_3$)-free Digraphs with Unbounded Dichromatic Number}

\author{Guillaume Aubian}
\address{CRED, Université Paris-Panthéon-Assas}

\author{Luis Kuffner}
\address{Université Paris Cité, CNRS, IRIF, F-75013 Paris, France}

\date{November 2025}

\begin{document}
\maketitle

\begin{abstract}
We construct orientations of rook graphs (whose underlying graphs are claw-free) that contain no $\vec C_3$ but have unbounded dichromatic number. This disproves a conjecture of Aboulker, Charbit and Naserasr and improves a result of Carbonero, Koerts, Moore and Spirkl.
\end{abstract}

\section{Introduction and notation}

A hero is a tournament $H$ such that tournaments not containing $H$ have bounded dichromatic number. In a seminal paper \cite{BCCFLSST13}, Berger, Choromanski, Chudnovski, Fox, Loebl, Scott, Seymour and Thomassé obtained a structure theorem for heroes. Harutyunyan, Le, Newman and Thomassé \cite{HLNT17} later proved that oriented graphs with bounded independence number (tournaments being the special case where the independence number equals one) not containing a hero have bounded dichromatic number. Chudnovsky, Scott and Seymour \cite{CSS19} subsequently showed that for any transitive tournament $H$ and oriented star $S$, oriented graphs containing neither $S$ nor $H$ as induced subdigraphs have bounded dichromatic number. Aboulker, Charbit and Naserasr \cite{ACN21} thus wondered whether (and actually conjectured that), for any hero $H$ and any oriented star forest $S$, oriented graphs containing neither $S$ nor $H$ as induced subdigraphs have bounded dichromatic number.

\smallskip{}

Aboulker, Aubian and Charbit \cite{AAC24} provided a counterexample when $S = K_1 + \vec K_2$. Carbonero, Koerts, Moore and Spirkl \cite{CKMS25} did so when $H = \Delta(1,1,C_3)$ and $S$ is any orientation of $K_{1,4}$, and also when $H = \vec C_3$ and $S$ is any orientation of $K_{1,5}$. We improve these two latter results with a counterexample when $H = \vec C_3$ and $S$ is any orientation of $K_{1,3}$, i.e. the claw.

\bigskip{}

Definitions in this paper are standard and follow from Bondy and Murty's book \cite{BM08}. In this paper, we study oriented graphs, i.e. directed graphs with no pair of opposite arcs. Given two integers $n,i \ge 0$, we write $n_i$ for the $i$th least significant bit of the binary expansion of $n$, starting from $i=0$.

\section{Construction and proof}

Let $N$ be a positive integer. The \emph{$N \times N$ rook graph} $R_N$ is the graph with vertex set $[1,N]^2$ and with an edge between two vertices whenever they share their first coordinate or they share their second coordinate. Intuitively, it is the graph describing how a rook moves on a $N \times N$ chessboard. It is the line graph of $K_{N,N}$, and therefore is claw-free. We orient it into a digraph $D_N$ as follows. For any vertices $(a,b),\ (a,d),\ (c,b)$, we orient the edge from $(a,b)$ to $(a,d)$ if and only if $b_i = a_i$, where $i = \min\{ j \mid b_j \ne d_j \}$, and we orient the edge from $(a,b)$ to $(c,b)$ if and only if $b_i \ne a_i$, where $i = \min\{ j \mid a_j \ne c_j \}$.

\begin{proposition}
The digraph $D_N$ contains no $\vec C_3$.
\end{proposition}
\begin{proof}
Given three vertices $(a,b)$, $(a,c)$ and $(a,d)$ on the same row, let $i = \min\{j \mid \{b_j, c_j, d_j\} = \{0,1\}\}$. Without loss of generality let us assume that $b_i \ne c_i = d_i$. Edges $(a,b)(a,c)$ and $(a,b)(a,d)$ are oriented only depending on whether $a_i = b_i$ and thus have the same orientation. Thus, $(a,b)$, $(a,c)$ and $(a,d)$ can not form a $\vec C_3$. The same argument applies to any column. Since every triangle in a rook graph is contained entirely in a single row or in a single column, $D_N$ contains no $\vec C_3$.
\end{proof}

\begin{lemma}
When $|a - c| = |b - d|$, vertices $(a,b),\ (a,d),\ (c,d),\ (c,b)$ induce a directed $4$-cycle.
\end{lemma}
\begin{proof}
Let $i=\min \{j \mid |a - c|_{j}  = 1\}$. Then $i = \min\{j \mid a_j \ne c_j\} = \min\{j \mid b_j \ne d_j\}$.
Up to permuting $a$ and $c$ and/or permuting $b$ and $d$, we may assume $a_i = b_i = 1$ and $c_i = d_i = 0$.
Since $b_i = a_i$, we have $(a,b) \to (a,d)$. Since $d_i \ne a_i$, we have $(a,d) \to (c,d)$. Since $d_i = c_i$, we have $(c,d) \to (c,b)$. Since $b_i \ne c_i$, we have $(c,b) \to (a,b)$.
Thus the four vertices form the directed cycle $(a,b) \to (a,d) \to (c,d) \to (c,b) \to (a,b)$.
\end{proof}

Note that it actually suffices that $d_2(a,c) = d_2(b,d)$ where $d_2$ is the dyadic distance, but this stronger condition suffices for the sake of our result.

\begin{theorem}
The family $(D_N)_{N\ge 1}$ has unbounded dichromatic number.
\end{theorem}

\begin{proof}
Fix $k$. The Gallai--Witt theorem  \cite{W51} (also known as the multidimensional van der Waerden theorem) implies that there exists $N$ such that any $k$-colouring of $[1,N]^2$ contains a monochromatic axis-parallel square, that is, four points of the form $(x,y),\ (x+t,y),\ (x,y+t),\ (x+t,y+t)$ for some integers $x,y,t>0$. Thus, for $N$ large enough, any colouring of $V(D_N)$ with $k$ colours contains a monochromatic square as above. By the previous lemma, the four vertices of such a square induce a directed $4$-cycle. Hence, some colour class does not induce an acyclic subdigraph. Since $k$ is arbitrary, the dichromatic number of $D_N$ is unbounded.
\end{proof}

\bibliographystyle{amsplain}
\bibliography{main}

@article{HLNT17,
  TITLE = {{Coloring dense digraphs}},
  AUTHOR = {Harutyunyan, Ararat and Le, Tien-Nam and Newman, Alantha and Thomass{\'e}, St{\'e}phan},
  URL = {https://hal.science/hal-03943075},
  JOURNAL = {{Electronic Notes in Discrete Mathematics}},
  PUBLISHER = {{Elsevier}},
  VOLUME = {61},
  PAGES = {577-583},
  YEAR = {2017},
  MONTH = Aug,
  DOI = {10.1016/j.endm.2017.07.010},
  HAL_ID = {hal-03943075},
  HAL_VERSION = {v1},
}

@article{CKMS25,
title = {On heroes in digraphs with forbidden induced forests},
journal = {European Journal of Combinatorics},
volume = {125},
pages = {104104},
year = {2025},
issn = {0195-6698},
doi = {https://doi.org/10.1016/j.ejc.2024.104104},
url = {https://www.sciencedirect.com/science/article/pii/S0195669824001896},
author = {Alvaro Carbonero and Hidde Koerts and Benjamin Moore and Sophie Spirkl}
}

@article{ACN21,
  title        = {Extension of Gyárfás--Sumner Conjecture to Digraphs},
  author       = {Pierre Aboulker and Pierre Charbit and Reza Naserasr},
  journal      = {The Electronic Journal of Combinatorics},
  volume       = {28},
  number       = {2},
  pages        = {P2.27},
  year         = {2021},
  doi          = {10.37236/9906},
  url          = {https://www.combinatorics.org/ojs/index.php/eljc/article/view/v28i2p27}
}

@article{AAC24,
author = {Aboulker, Pierre and Aubian, Guillaume and Charbit, Pierre},
title = {Heroes in oriented complete multipartite graphs},
journal = {Journal of Graph Theory},
volume = {105},
number = {4},
pages = {652-669},
doi = {https://doi.org/10.1002/jgt.23061},
url = {https://onlinelibrary.wiley.com/doi/abs/10.1002/jgt.23061},
eprint = {https://onlinelibrary.wiley.com/doi/pdf/10.1002/jgt.23061},
year = {2024}
}

@article{W51,
  author = {Witt, Ernst},
  title = {Ein kombinatorischer Satz der Elementargeometrie},
  journal = {Mathematische Nachrichten},
  volume = {6},
  pages = {261--262},
  year = {1951}
}

@book{BM08,
author = {Bondy, J.A. and Murty, U.S.R},
title = {Graph Theory},
year = {2008},
isbn = {1846289696},
publisher = {Springer Publishing Company, Incorporated},
edition = {1st}
}

@article{CSS19,
title = {Induced subgraphs of graphs with large chromatic number. XI. Orientations},
journal = {European Journal of Combinatorics},
volume = {76},
pages = {53-61},
year = {2019},
issn = {0195-6698},
doi = {https://doi.org/10.1016/j.ejc.2018.09.003},
url = {https://www.sciencedirect.com/science/article/pii/S0195669818301562},
author = {Maria Chudnovsky and Alex Scott and Paul Seymour}
}

@article{BCCFLSST13,
title = {Tournaments and colouring},
journal = {Journal of Combinatorial Theory, Series B},
volume = {103},
number = {1},
pages = {1-20},
year = {2013},
issn = {0095-8956},
doi = {https://doi.org/10.1016/j.jctb.2012.08.003},
url = {https://www.sciencedirect.com/science/article/pii/S0095895612000664},
author = {Eli Berger and Krzysztof Choromanski and Maria Chudnovsky and Jacob Fox and Martin Loebl and Alex Scott and Paul Seymour and Stéphan Thomassé}
}

\end{document}